\documentclass{proc-l}

\usepackage{lineno,hyperref}
\usepackage{amssymb,enumerate}
\usepackage{amsmath}
\usepackage{graphics}
\numberwithin{equation}{section}
\def\p{\partial}
\def\te{\textbf}

\def\O{{\bf O}}

\def\p{\partial}

\newtheorem{thm}{Theorem}[section]

\newtheorem{lem}{Lemma}[section]

\newtheorem{rem}{Remark}[section]

\numberwithin{equation}{section}

\begin{document}

\title{Formation of finite-time singularities for nonlinear elastodynamics with small initial disturbances}


%    author one information
\author{Zhentao Jin}
\address{School of Mathematical Sciences, Fudan University, Shanghai 200433, China.}
\email{ztjin18@fudan.edu.cn}
\thanks{}

%    author two information
\author{Yi Zhou}
\address{School of Mathematical Sciences, Fudan University, Shanghai 200433, China, and Department of Mathematical Sciences, Jinan University, Guangzhou 510632, China.}
\email{yizhou@fudan.edu.cn}
\thanks{}

%    \subjclass is required.
\subjclass[2010]{Primary }

\date{}

\dedicatory{}

%    "Communicated by" -- provide editor's name; required.

%    Abstract is required.
\begin{abstract}
This article concerns the formation of finite-time singularities in solutions to  quasilinear hyperbolic systems with small initial data. By constructing a special test function, we first present a simpler proof of the main result in \cite{sideris1985formation}:  the global classical solution is non-existent for compressible Euler equation even for some small initial data. Then we apply this approach to nonlinear elastodynamics and magnetohydrodynamics, showing that the classical solutions to these equations can still blow up in finite time even if the initial data is small enough.
\end{abstract}

\maketitle
%%%%%%%%%%%%%%%%%%%%%%%%%%%%%%%%%%%%%%%%%%%%%%%%%%
\section{Introduction}
\par\quad
The Cauchy problem of quasilinear hyperbolic systems plays an important role in modern PDEs since many physical phenomena can be described by these systems, the Euler equations for ideal fluid, elastodynamics for elastic material and magnetohydrodynamics are all typical examples. Because singularity is fairly common for these systems, sometimes it is infeasible to find a global regular solution even though the initial data is smooth and small enough. As a consequence, the study of formation of singularities becomes a significant part of this area.

The general one-dimensional theory \cite{john1974formation,klainerman1980formation,lax1964development,liu1979development,LiTasien1994global} for these systems was established thanks to the method of characteristics, while in higher dimensions this powerful method is no longer effective. In 1985, by using averaged quantities, T. C. Sideris \cite{sideris1985formation} proved that there exist smooth initial data, such that the classical global solution to 3D compressible Euler equation is non-existent. His paper included two results, one for large initial data while another for small data. After that, the large data results were extended to relativistic fluid \cite{guo1999formation}, Euler-Poisson equation \cite{makino1990solution,perthame1990non}, Euler-Maxwell equation \cite{guo1999formation}, magnetohydrodynamics \cite{rammaha1989formation,rammaha1994formation}, elastodynamics \cite{tahvildar1998relativistic} etc. Later, for small initial data, S. Alinhac \cite{alinhac2001null1,alinhac2001null2} and D. Christodoulou \cite{christodoulou2014compressible} presented more complete description of singularity in dimensions of two and three respectively. Although their tools are effective for Euler equation, it seems that there still exist some serious difficulties to employ these approaches to elastodynamics and magnetohydrodynamics.

In this paper, we introduce a special test function different from \cite{sideris1985formation} to consider formation of singularities for Euler equations, elastodynamics and magnetohydrodynamics with small initial data. Our results show that the classical solutions to these systems have to break down in finite time. Under spherical symmetric assumption, F. John \cite{John1985Blow} proved 3D compressible elastodynamics¡¯ classical solution could break down in finite time even if the initial data is small enough. But for asymmetric case, there are only results under certain large data assumptions, our proof gives some small-initial-data results for this case. Besides, the natural structure of our test function makes our calculation much simpler.

This paper is organized as follows: In Section 2, we introduce our test function and give an ode lemma . Section 3 includes our main theorems, which indicates the blow-up results for Euler equation, elastodynamics and magnetohydrodynamics respectively.
\section{Test function}
Definite test function $F(x),\ x\in\mathbb{R}^n$ as follows:
\begin{equation}\label{testfunction}
\ F(x)=\left \{
\begin{aligned}
&e^x+e^{-x}\ \ \ &n=1,\\
&\int_{S^{n-1}}e^{\omega\cdot x}d\omega\ \ \ &n\geq 2.\\
\end{aligned} \right.
\end{equation}
we can conclude that:
\begin{align}
& (i).\ F(x)>0,\ \Delta F(x)=F(x),\\
& (ii).\ F(x)=F(|x|),\\
& (iii).\ F(x)\simeq (1+|x|)^{-\frac{n-1}{2}}e^{|x|}.
\end{align}
the property (i) and (ii) means $F(x)$ is positive and symmetric, while (iii) points out the growth of this function. The first two properties are obvious and the proof of (iii) can be found in \cite{LiTasien2017}.

Then, we need to introduce an ode lemma which will be frequently used in this paper:
\begin{lem}\label{OOO}
Let X(t) be a smooth function satisfies the following inequalities:
\begin{equation}\label{LeODE}
X^{''}(t)\geq \frac{CX^2(t)}{e^t(t+R_0)^{\frac{n-1}{2}}}+X(t),\ \ t\geq 0,\ 1\leq n\leq 3.
\end{equation}
with initial data $X(0)+X^{'}(0)=\varepsilon x_0\geq 0$ is small enough,\ then X(t) will blow up in finite time. Furthermore, the lifespan of $X(t)$ is:
\begin{equation}\label{lifespan}
\ T_0=\left \{
\begin{aligned}
&C\varepsilon^{-1}\ \ \ &n=1,\\
&C\varepsilon^{-2}\ \ \ &n=2,\\
&e^{C\varepsilon^{-1}}\ \ \ &n=3.\\
\end{aligned} \right.
\end{equation}
\end{lem}
\begin{proof}
Without loss of generality, we consider $R_0=1,\ n=3$ only. Take $X(t)=e^tZ(t)$, the inequality becomes:
\begin{equation}
Z^{''}+2Z^{'}\geq \frac{CZ^2}{t+1}
\end{equation}
Next, let $t+1=e^\tau $, or $\tau=\ln (t+1)$, then:
\begin{align*}
\frac{dZ}{dt}=e^{-\tau}\frac{dZ}{d\tau};\ \ \frac{d^2Z}{dt^2}=e^{-\tau}\frac{d}{d\tau}(e^\tau\frac{dZ}{d\tau})
\end{align*}
and
\begin{equation}\label{ODEttau}
\frac{d}{d\tau}(e^\tau\frac{dZ}{d\tau})+2\frac{dZ}{d\tau}\geq CZ^2
\end{equation}
Assume $Z(\tau)$ exists in $[0,T]$, define cut-off function $\chi(\tau)\in C^{\infty}([0,+\infty))$ as follows:
\begin{equation}
\chi(\tau)=\left \{
\begin{aligned}
&1\ \   \ \ &\tau<\frac{1}{8}\\
&\in[0,1] \ \  &\frac{1}{8}<\tau<\frac{7}{8}\\
&0\ \ \ \ &\tau>\frac{7}{8}.
\end{aligned} \right.
\end{equation}
Multiply the inequality \eqref{ODEttau} by $\chi^4(\frac{\tau}{T})$ and integration from 0 to T:
\begin{align}
\int_0^T(\frac{d}{d\tau}(e^{-\tau}\frac{dZ}{d\tau})+2\frac{dZ}{d\tau})\chi^4(\frac{\tau}{T})d\tau\geq C\int_0^T\chi^4(\frac{\tau}{T})Z^2d\tau
\end{align}
After that, using integration by parts, notice that $\chi(1)=0,\ \chi^{'}(0)=0$, we have:
\begin{align*}
&-4\int_0^Te^{-\tau}Z(\tau)\chi^3(\frac{\tau}{T})\chi^{'}(\frac{\tau}{T})\frac{1}{T}d\tau+12\int_0^Te^{-\tau}Z(\tau)\chi^2(\frac{\tau}{T})(\chi^{'}(\frac{\tau}{T}))^2\frac{1}{T^2}d\tau\\
&+4\int_0^Te^{-\tau}Z(\tau)\chi^3(\frac{\tau}{T})\chi^{''}(\frac{\tau}{T})\frac{1}{T^2}d\tau-8\int_0^TZ(\tau)\chi^3(\frac{\tau}{T})\chi^{'}(\frac{\tau}{T})\frac{1}{T}d\tau\\
&\geq Z^{'}(0)+2Z(0)+C\int_0^T\chi^4(\frac{\tau}{T})Z^2(\tau)d\tau
\end{align*}
By Cauchy-Schwartz inequality, since $\chi(\tau),\chi^{'}(\tau)$ is bounded, get:
\begin{align}
Z^{'}(0)+2Z(0)+C\int_0^T\chi^4(\frac{\tau}{T})Z^2(\tau)d\tau&\leq \tilde C(\frac{1}{T^{\frac{3}{2}}}+\frac{1}{T^\frac{1}{2}})(\int_0^T\chi^4(\frac{\tau}{T})Z^2(\tau)d\tau)^{\frac{1}{2}}\notag\\
&\leq \bar{C}(\frac{1}{T}+\frac{1}{T^3})+C\int_0^T\chi^4(\frac{\tau}{T})Z^2(\tau)d\tau
\end{align}

Since $Z(0)=X(0),\ Z^{'}(0)=-X(0)+X^{'}(0)$, we have $T<\frac{\bar C}{X(0)+X^{'}(0)}=\bar C\varepsilon^{-1}$. And $t+1=\tau$, so the life span of X(t) is $e^{C\epsilon^{-1}}$, X(t) will blow up in finite time.

The case $n=1,2$ can be treated in a similar way.
\end{proof}
\section{Main results}
\subsection{Compressible Euler equation}
First of all, we give a different proof of Sideris' theorem by using our test function $F(x)$. The following Euler equation represents the state of the inviscid compressible flow in three dimensions:

\begin{equation}\label{Euler}
\left \{
\begin{aligned}
&\rho_t+div(\rho \textbf{u})=0,\\
&(\rho\textbf{u})_t+div(\rho\textbf{u}\otimes\textbf{u})+\nabla p=0,\\
&t=0:\ \ \textbf{u}=\textbf u_0(x),\  \rho=1+\rho_0(x). \\
\end{aligned} \right.
\end{equation}
where $\rho,\textbf{u}$ and $p$ are density, velocity and pressure. Without loss of generality, we can assume $p(\rho)=\frac{1}{2}\rho^2$, then the speed of sound $c^2=p^{'}(\rho)=\rho$.  Besides, the initial data $\textbf u_0(x),\rho_0(x)$ satisfies:
\begin{align}\label{Eulerinitialdata}
\textbf u_0(x),\rho_0(x)&\in C_0^\infty(\mathbb{R}^3),\ \\
supp\  \textbf u_0(x),\rho_0(x)&\subset \{x| |x|\leq R_0.\}\notag
\end{align}
these conditions ensure that $supp\ \textbf u(t,x),\rho(x)\subset \{x| |x|\leq R_0+t\}$. For the convenience, always take $R_0=1$.
\begin{thm}\label{Euletthm}
Suppose $(\rho,\textbf u)$ is a classical solution of \eqref{Euler}\eqref{Eulerinitialdata} for $0<t<T$. If the initial data satisfies:
 \begin{equation}
 \int_{\mathbb R^3}\int_{S^2}e^{\omega\cdot x}\rho_0\ d\omega dx+\int_{\mathbb R^3}\int_{S^2}e^{\omega\cdot x}(1+\rho_0)(\textbf u_0\cdot \omega)\ d\omega dx>0
 \end{equation}
then the life span T is less than $e^{C\varepsilon^{-1}}$.
\end{thm}
\begin{proof}
Multiply \eqref{Euler} by $e^{\omega\cdot x}$ and integration by parts to get:
\begin{align}\label{Euler_1}
&\int_{\mathbb R^{3}}\int_{S^2}(\rho_t+\nabla(\rho\textbf u))e^{\omega\cdot x}\ d\omega dx\\
=&\frac{d}{dt}\int_{\mathbb R^{3}}\int_{S^2}\rho e^{\omega\cdot x}\ d\omega dx-\int_{\mathbb R^{3}}\int_{S^2}\rho e^{\omega\cdot x}(\textbf u \cdot {\omega})\ d\omega dx\notag\\
=&0\notag
\end{align}
Next, multiply \eqref{Euler} by $e^{\omega\cdot x}\omega$, get:
\begin{align}\label{Euler_2}
&\int_{\mathbb{R}^3}\int_{S^2}e^{\omega\cdot x}(\sum_{i=1}^{3}\frac{\partial(\rho u_i \omega_i)}{\partial t}+\sum_{i,k=1}^{3}\frac{\partial(\rho u_iu_k\omega_i)}{\partial x_k}+\sum_{i=1}^{3}\omega_i\frac{\partial p(\rho)}{\partial x_i})\ d\omega dx\\
=&\frac{d}{dt}\int_{\mathbb{R}^3}\int_{S^2}e^{\omega\cdot x}\rho(\textbf u\cdot \omega)\ d\omega dx-\int_{\mathbb{R}^3}\int_{S^2}e^{\omega\cdot x}\rho(\textbf u\cdot \omega)^2\ d\omega dx\notag\\
-&\frac{1}{2}\int_{\mathbb{R}^3}\int_{S^2}(\rho^2-1)e^{\omega\cdot x}\ d\omega dx\notag\\
=&0\notag
\end{align}
If we take:
\begin{align}
&X(t)=\int_{\mathbb R^3}F(x)(\rho(t,x)-1)\ dx=\int_{\mathbb R^3}\int_{S^2}e^{\omega\cdot x}(\rho(t,x)-1)\ d\omega dx\\
&Y(t)=\int_{\mathbb R^3}F(x)\rho(t,x)(\textbf u\cdot\omega)\ dx=\int_{\mathbb R^3}\int_{S^2}e^{\omega\cdot x}\rho(t,x)(\textbf u\cdot \omega)\ d\omega dx
\end{align}
then from \eqref{Eulerinitialdata}, \eqref{Euler_1} and \eqref{Euler_2} the relation between X(t) and Y(t) follows:
\begin{align}\label{X(t)Y(t)}
X^{'}(t)&=Y(t),\\
Y^{'}(t)&=\int_{|x|\leq t+1}\int_{S^2}e^{\omega\cdot x}\rho(\textbf u\cdot \omega)^2\ d\omega dx+\frac{1}{2}\int_{|x|\leq t+1}\int_{S^2}(\rho^2-1)e^{\omega\cdot x}\ d\omega dx\\
&=\int_{|x|\leq t+1}\int_{S^2}e^{\omega\cdot x}\rho(\textbf u\cdot \omega)^2\ d\omega dx+\frac{1}{2}\int_{|x|\leq t+1}\int_{S^2}(\rho-1)^2e^{\omega\cdot x}\ d\omega dx\notag\\
&+\int_{|x|\leq t+1}\int_{S^2}(\rho-1)e^{\omega\cdot x}\ d\omega dx\notag\\
&\geq \frac{1}{2}\int_{|x|\leq t+1}\int_{S^2}(\rho-1)^2e^{\omega\cdot x}\ d\omega dx+X(t)\notag
\end{align}
By H\"{o}lder inequality:
\begin{align*}
&X(t)\leq (\int_{|x|\leq t+1}\int_{S^2}e^{\omega\cdot x}(\rho-1)^2\ d\omega dx)^{\frac{1}{2}}\cdot(\int_{|x|\leq t+1}\int_{S^2}e^{\omega\cdot x}d\omega dx)^{\frac{1}{2}}
\end{align*}
that is:
\begin{align}
\int_{|x|\leq t+1}\int_{S^2}e^{\omega\cdot x}(\rho^2-1)\ d\omega dx \geq \frac{X^2}{\int_{|x|\leq t+1}\int_{S^2}e^{\omega\cdot x}d\omega dx}
\end{align}
Since:
\begin{align}
&\int_{|x|\leq t+1}\int_{S^2}e^{\omega\cdot s}d\omega dx=\int_{|x|\leq t+1}F(x)\ dx \\
\leq& C\int_0^{t+1}e^r(r+1)^{-1}r^{n-1}\ dr \notag\\
\leq& Ce^{t+1}(t+1)\int_0^{t+1}e^{r-t-1}\ dr\notag\\
\leq& Ce^{t+1}(t+1).\notag
\end{align}
Finally, we can get an ordinary differential inequality:
\begin{equation}\label{ODE}
\left \{
\begin{aligned}
&X^{'}(t)=Y(t),\\
&Y^{'}(t)\geq \frac{CX^2(t)}{e^t(t+1)}+X(t).\\
\end{aligned} \right.
\end{equation}

Thus, from lemma \ref{OOO}, we complete the proof of theorem \eqref{Euletthm}.
\end{proof}
\begin{rem}
We can give another simpler proof of this theorem. In fact, decompose $x=(y,z),\textbf u=(\textbf v,\textbf w)$, in which $y,\textbf v\in \mathbb{R}^d,\ z,\textbf w\in \mathbb{R}^{n-d}$ $(1\leq d \leq n-1)$, the compressible Euler equations becomes:
\begin{equation}\label{EulerDec}
\left \{
\begin{aligned}
&\rho_t+\nabla_y\cdot(\rho \textbf{v})+\nabla_z\cdot(\rho \textbf{w})=0,\\
&\partial_t(\rho \textbf v)+\nabla_y\cdot(\rho \textbf v\otimes \textbf v)+\nabla_z\cdot(\rho \textbf v\otimes\textbf w)+\nabla_y(p(\rho))=0,\\
&\partial_t(\rho \textbf w)+\nabla_z\cdot(\rho \textbf w\otimes \textbf w)+\nabla_y\cdot(\rho \textbf w\otimes\textbf v)+\nabla_z(p(\rho))=0.\\
\end{aligned} \right.
\end{equation}

For simplicity, take d=1. Follow the same steps, let F(y),\ X(t),\ Y(t) be:
\begin{align}
& F(y)=e^y+e^{-y},\\
& X(t)=\int_{\mathbb{R}^{n-1}}\int_\mathbb{R}(\rho(t,y,z)-1)(e^y+e^{-y})\ dydz,\\
& Y(t)=\int_{\mathbb{R}^{n-1}}\int_\mathbb{R}\rho(t,y,z) v(t,y,z)(e^y+e^{-y})\ dydz.
\end{align}
Consider the first two equations of \eqref{EulerDec} only, multiply $e^y+e^{-y}$ and $e^y-e^{-y}$ respectively, then integration by parts,since that:
\begin{align*}
&\int\int_{y^2+z^2\leq (t+1)^2}F(y)\ dz dy=\int\int_{y^2+z^2\leq (t+1)^2}e^y+e^{-y}\ dz dy\\
=&2\int_0^{t+1}\int_{|z|\leq\sqrt{(t+1)^2}-y^2} e^y\ dz dy\leq C\int_0^{t+1}e^y((t+1)^2-y^2)^{\frac{n-1}{2}}dy\\
\leq& Ce^{t+1}(t+1)^{\frac{n-1}{2}}\int_0^{t+1}e^{y-t-1}(t+1-y)^{\frac{n-1}{2}}dy\\
=& Ce^{t+1}(t+1)^{\frac{n-1}{2}}\int_0^{t+1}e^{-r}r^{\frac{n-1}{2}}dr\leq Ce^{t+1}(t+1)^{\frac{n-1}{2}}
\end{align*}

Then, we can get the similar inequality:
\begin{equation}
X^{''}(t)\geq X(t)+\frac{CX^2(t)}{e^{t+1}(t+1)^{\frac{n-1}{2}}}
\end{equation}
When $1\leq n\leq 3$, with appropriate small initial data, $X(t)$ blows up in finite time.
\end{rem}

Furthermore, if there is viscosity in z direction only, this method can still work.
\begin{thm}
Consider the following initial boundary value problem for the compressible Navier-Stokes equations without vertical viscosity:
\begin{equation*}
\left \{
\begin{aligned}
&\rho_t+\nabla_y\cdot(\rho \textbf{v})+\nabla_z\cdot(\rho \textbf{w})=0,\\
&\partial_t(\rho \textbf v)+\nabla_y\cdot(\rho \textbf v\otimes \textbf v)+\nabla_z\cdot(\rho \textbf v\otimes\textbf w)+\nabla_y(p(\rho))=0,\\
&\partial_t(\rho \textbf w)+\nabla_z\cdot(\rho \textbf w\otimes \textbf w)+\nabla_y\cdot(\rho \textbf w\otimes\textbf v)+\nabla_z(p(\rho))=\lambda\Delta_z w+\mu\nabla_z(\nabla_z\cdot \textbf w),\\
&t=0:\ \textbf v=\textbf v_0(y,z),\ \textbf w=\textbf w_0(y,z),\  \rho=1+\rho_0(y,z).
\end{aligned} \right.
\end{equation*}
where $y\in \mathbb{R}^d,\ z\in\Omega\subset\mathbb{R}^{n-d},\ \textbf v\in \mathbb{R}^d,\ \textbf w\in R^{n-d}$, $1\leq d \leq n-1\leq 3$, $\Omega$ is a bounded domain. If the initial data of this system satisfies:
\begin{align}
&\textbf v_0(y,z),\textbf w_0(y,z),\rho_0(y,z)\in C^\infty(\mathbb R^d\times \Omega),\\
&supp\  \textbf v_0(y,z),\textbf w_0(y,z),\rho_0(y,z)\subset \{y| |y|\leq 1\}.
\end{align}
Besides, $\te w$ satisfies Dirichlet boundary condition:
\begin{equation}
\textbf w(y,z)|_{z\in \partial \Omega}=0.
\end{equation}
Then if the initial data satisfies:
\begin{equation}\label{331}
\int_\Omega\int_{\mathbb{R}^d}\int_{S^{d-1}}\rho_0e^{\omega\cdot y}\ d\omega dydz+\int_\Omega\int_{\mathbb{R}^d}\int_{S^{d-1}}(1+\rho_0) (\textbf v_0\cdot\omega)\ d\omega dydz>0
\end{equation}
the classical solution to the equation above blows up in finite time.
\end{thm}
\begin{proof}
Again, let us define $F(y),\ X(t),\ Y(t)$ :
\begin{align}
& F(y)=\int_{S^{d-1}}e^{\omega\cdot y}\ d\omega,\\
& X(t)=\int_\Omega\int_{\mathbb{R}^d}\int_{S^{d-1}}(\rho(t,y,z)-1)e^{\omega\cdot y}\ d\omega dydz,\\
& Y(t)=\int_\Omega\int_{\mathbb{R}^d}\int_{S^{d-1}}\rho(t,y,z) (\textbf v(t,y,z)\cdot\omega)\ d\omega dydz.
\end{align}
Multiply $e^{\omega\cdot y}$ and $e^{\omega\cdot y}\omega$ to the first two equations, integrate $\omega,\ y,\ z$. Thanks to the Dirichlet boundary condition $\textbf w(y,z)|_{z\in \partial \Omega}=0$, all the boundary term disappear, we can get the ordinary differential inequality as usual:
\begin{equation}
\left \{
\begin{aligned}
&X^{'}(t)=Y(t),\\
&Y^{'}(t)\geq \frac{CX^2(t)}{e^t(t+1)^{\frac{d-1}{2}}|\Omega|}+X(t).\\
\end{aligned} \right.
\end{equation}

Therefore, for initial data \eqref{331}, classical solution still blows up in finite time.
\end{proof}
\subsection{Compressible elastodynamics}
We consider the displacement of an isotropic, homogeneous, hyperelastic material, set Lagrange function as follows:
\begin{equation}
\mathcal{L}(\textbf u)=\int\int \frac{1}{2}|\textbf u_t|^2-W(\nabla \textbf u)\ dxdt.
\end{equation}

In this definition, $\textbf u$ represents the displacement from the reference configuration while $W(\nabla\textbf u)$ means stored energy function. Since material is isotropic, that is:
\begin{align}
W(F)=W(QF),\ \forall Q\in SO(3);\\
W(F)=W(FQ),\ \forall Q\in SO(3).
\end{align}

Using Hamilton's principle, the corresponding Euler-Lagrange equation is:
\begin{equation}
\frac{\p^2 u^i}{\p t^2}-\p_k(\frac{\p W}{\p u^i_{x^k}}(\nabla \te u))=0
\end{equation}

Ignore cubic and higher order terms, the motion for these materials can be described by:
\begin{equation}\label{elastic}
\p_t^2\te u-c_2^2\Delta\te u-(c_1^2-c_2^2)\nabla(\nabla\cdot\te u)=F(\nabla\te u,\nabla^2\te u)\ \ .
\end{equation}

In which the material constants $c_1,\ c_2\ (c_1>c_2>0)$ correspond to the propagation speeds of longitudinal and transverse waves, and F is nonlinear term, without loss of generality, always assume $c_1=1$.

For three dimensional isotropic hyperelastic, we can expand nonlinear term explicitly (see \cite{agemi2000global}):
\begin{align}
F(\nabla\te u,\nabla^2\te u)
&=2(2\sigma_{111}+3\sigma_{11})\nabla(\nabla\cdot \textbf u)^2+2(\sigma_{11}-\sigma_{12})\nabla|\nabla\times\textbf u|^2\\
&-4(\sigma_{11}-\sigma_{12})\nabla\times(\nabla\cdot\textbf u\nabla\times\textbf u)+Q(\textbf u,\nabla \textbf u)\notag\\
& \notag\\
Q^i(\textbf u,\nabla \textbf u)&=Q^i_1(\textbf u,\nabla \textbf u)+Q^i_2(\textbf u,\nabla \textbf u)\\
& \notag\\
Q^i_1(\textbf u,\nabla \textbf u)&=4(2\sigma_{12}-\sigma_{11})\sum_{j,k=1}^3(Q_{jk}(u^j,\partial_iu^k)-Q_{ik}(u^k,\partial_ju^j))\\
Q^i_2(\textbf u,\nabla \textbf u)&=2(\sigma_2-\sigma_3)\sum_{j,k=1}^3(2Q_{jk}(\partial_ju^i,u^k)+Q_{jk}(\partial_ku^j,u^i)+Q_{ij}(\partial_ju^k,u^k))\\
&+2\sigma_3\sum_{j,k=1}^3(2Q_{ij}(\partial_ku^k,u^j)+Q_{ji}(\partial_ku^j,u^k))\notag\\
here\ Q_{ij}(f,g)&=\partial_if\partial_jg-\partial_ig\partial_jf,\ 4\sigma_{11}=c_1^2=1,\ -2\sigma_2=c_2^2<1.\notag
\end{align}

Take the initial data:
\begin{align}\label{elasticinitialdata}
t=0:\ \textbf u=&\varepsilon\textbf u_0(x),\ \textbf u_t=\varepsilon\textbf u_1(x)\\
\textbf u_0(x),\textbf u_1(x)\in C_0^\infty,\ \ &supp\ \textbf u_0,\textbf u_1(x)\subset \{x||x|\leq 1\}\notag
\end{align}
\begin{thm}
Consider 3D compressible elastodynamics \eqref{elastic}\eqref{elasticinitialdata}, if the structural constants satisfy:
\begin{align}\label{sigmaconditions}
\sigma_{111}=-O(\lambda^2)<0,\ \sigma_{12}=O(\lambda)>0,\ |\sigma_3|<<\lambda.
\end{align}
where $\lambda>>1$ is a constant. Then with the following initial data:
\begin{equation}\label{3311}
\int_{\mathbb{R}^3}\int_{S^2}e^{\omega\cdot x} (\te u_0+\te u_1)\cdot\omega d\omega dx>0
\end{equation}
the classical solution to this equation will blow up in finite time and the life-span is $e^{C\varepsilon^{-1}}$, C depends on $\lambda$.
\end{thm}
\begin{proof}
Multiply equation \eqref{elastic} by $e^{\omega\cdot x}\omega^i$, sum over i and integration in $\mathbb{R}^3\times S^2$:
\begin{align}\label{elasticie}
&\frac{d^2}{dt^2}\int_{\mathbb{R}^3}\int_{S^2}e^{\omega\cdot x}(\textbf u\cdot \omega)\ d\omega dx\\
=&\int_{\mathbb{R}^3}\int_{S^2}e^{\omega\cdot x}(\textbf u\cdot \omega)\ d\omega dx
+\int_{\mathbb{R}^3}\int_{S^2}F(\nabla\te u,\nabla^2\te u)\cdot\omega e^{\omega\cdot x}\ d\omega dx.\notag
\end{align}
Substitute the expansion form of F, the last term in \eqref{elasticie} becomes:
\begin{align}\label{W}
&\int_{\mathbb{R}^3}\int_{S^2} F(\nabla\te u,\nabla^2\te u)\cdot \omega e^{\omega\cdot x}\ d\omega dx\\
=&-\int_{\mathbb{R}^3}\int_{S^2}2(2\sigma_{111}+3\sigma_{11})|\nabla\cdot\textbf u|^2e^{\omega\cdot x}\ d\omega dx\\
&-\int_{\mathbb{R}^3}\int_{S^2}2(\sigma_{11}-\sigma_{12})|\nabla\times\textbf u|^2e^{\omega\cdot x}\ d\omega dx\\
&-\int_{\mathbb{R}^3}\int_{S^2}4(\sigma_{11}-\sigma_{12})(\nabla\cdot\textbf u\nabla\times\textbf u)\cdot \nabla\times(\omega e^{\omega\cdot x})\ d\omega dx\\
&+\int_{\mathbb{R}^3}\int_{S^2}Q_1(\textbf u,\partial\textbf u)\cdot\omega e^{\omega\cdot x}\ d\omega dx\\
&+\int_{\mathbb{R}^3}\int_{S^2}Q_2(\textbf u,\partial\textbf u)\cdot\omega e^{\omega\cdot x}\ d\omega dx
\end{align}

In fact, $\nabla\times(\omega e^{\omega\cdot x})=0$, $(3.42)$ disappears. Then by calculation, we have the following identity:
\begin{align}\label{1}
&\int_{\mathbb{R}^3}\int_{S^2}| \nabla\textbf u|^2e^{\omega\cdot x}\ d\omega dx\\
=&\int_{\mathbb{R}^3}\int_{S^2}| \nabla\cdot\textbf u|^2e^{\omega\cdot x}\ d\omega dx
+\int_{\mathbb{R}^3}\int_{S^2}| \nabla\times\textbf u|^2e^{\omega\cdot x}\ d\omega dx\notag\\
+&\int_{\mathbb{R}^3}\int_{S^2} (\textbf u\cdot\omega)(\nabla\cdot\textbf u)e^{\omega\cdot x}\ d\omega dx
+\int_{\mathbb{R}^3}\int_{S^2}| \textbf u\cdot \omega|^2e^{\omega\cdot x}\ d\omega dx\notag\\
\leq&\frac{3}{2}\int_{\mathbb{R}^3}\int_{S^2}| \nabla\cdot\textbf u|^2e^{\omega\cdot x}\ d\omega dx
+\int_{\mathbb{R}^3}\int_{S^2}| \nabla\times\textbf u|^2e^{\omega\cdot x}\ d\omega dx\notag\\
+&\frac{3}{2}\int_{\mathbb{R}^3}\int_{S^2}| \textbf u\cdot \omega|^2e^{\omega\cdot x}\ d\omega dx\notag
\end{align}

\eqref{1} means that: the $L^2$ norm of $\nabla \textbf u$ with weight $e^{\omega\cdot x}$ can be controled by the same norm of $\nabla\cdot\textbf u,\ \nabla\times\textbf u,\ \textbf u\cdot\omega$ with the same weight. Since $\nabla\cdot\textbf u$ and $\nabla\times\textbf u$ appear in $(3.40)\ (3.41)$, in order to control $\textbf u\cdot\omega$, we need to consider $(3.43)$, divide $Q^i_1(\textbf u,\nabla \textbf u)$ into two parts: $Q_{jk}(u^j,\partial_iu^k)$ and $Q_{ik}(u^k,\partial_ju^j)$, deal with these integration in detil:
\begin{align}
&\int_{\mathbb{R}^3}\int_{S^2}Q_{jk}(u^j,\partial_iu^k)\cdot\omega^i e^{\omega\cdot x}\ d\omega dx\notag\\
=&\int_{\mathbb{R}^3}\int_{S^2}\partial_ju^j\partial_i\partial_ku^k\omega^ie^{\omega\cdot x}\ d\omega dx
-\int_{\mathbb{R}^3}\int_{S^2}\partial_ku^j\partial_j\partial_iu^k\omega^ie^{\omega\cdot x}\ d\omega dx\notag\\
=&\int_{\mathbb{R}^3}\int_{S^2}\partial_ju^j\partial_i\partial_ku^k\omega^ie^{\omega\cdot x}\ d\omega dx-\int_{\mathbb{R}^3}\int_{S^2}\partial_ju^j\partial_i\partial_ku^k\omega^ie^{\omega\cdot x}\ d\omega dx\notag\\
-&\int_{\mathbb{R}^3}\int_{S^2}\partial_ju^j\partial_iu^k\omega^i\omega^ke^{\omega\cdot x}\ d\omega dx
-\int_{\mathbb{R}^3}\int_{S^2} u^j\partial_i\partial_ku^k\omega_i\omega_je^{\omega\cdot x}\ d\omega dx\notag\\
-&\int_{\mathbb{R}^3}\int_{S^2}u^j\partial_iu^k\omega^i\omega^j\omega^ke^{\omega\cdot x}\ d\omega dx\notag
\end{align}
That is:
\begin{align}
&\int_{\mathbb{R}^3}\int_{S^2}Q_{jk}(u^j,\partial_iu^k)\cdot\omega^i e^{\omega\cdot x}\ d\omega dx\\
=&-\frac{1}{2}\int_{\mathbb{R}^3}\int_{S^2}|\nabla\cdot\textbf u|^2e^{\omega\cdot x}\ d\omega dx+\frac{1}{2}\int_{\mathbb{R}^3}\int_{S^2}|\nabla\cdot\textbf u|^2e^{\omega\cdot x}\ d\omega dx\notag\\
-&\int_{\mathbb{R}^3}\int_{S^2}(\nabla\cdot\textbf u)\partial_i(\textbf u\cdot \omega)\omega^i\ d\omega dx
-\int_{\mathbb{R}^3}\int_{S^2}\partial_i(\nabla\cdot\textbf u)(\textbf u\cdot \omega)\omega^ie^{\omega\cdot x}\ d\omega dx\notag\\
-&\int_{\mathbb{R}^3}\int_{S^2}\partial_i(\textbf u\cdot\omega)(\textbf u\cdot \omega)\omega^ie^{\omega\cdot x}\ d\omega dx\notag\\
=&\int_{\mathbb{R}^3}\int_{S^2}(\textbf u\cdot \omega)(\nabla\cdot \textbf u)e^{\omega\cdot x}+\frac{1}{2}\int_{\mathbb{R}^3}\int_{S^2}(\textbf u\cdot\omega)^2e^{\omega\cdot x}\ d\omega dx\notag\\
\geq& \frac{1}{4}\int_{\mathbb{R}^3}\int_{S^2}(\textbf u\cdot\omega)^2e^{\omega\cdot x}\ d\omega dx-\int_{\mathbb{R}^3}\int_{S^2}|\nabla\cdot \textbf u|^2e^{\omega\cdot x}\ d\omega dx\notag
\end{align}
Another term can be calculated in similar way:
\begin{align}
&\int_{\mathbb{R}^3}\int_{S^2}Q_{ik}(u^k,\partial_ju^j)\omega^ie^{\omega\cdot x}\ d\omega dx\\
=&-\int_{\mathbb{R}^3}\int_{S^2}\partial_i\partial_ku^k\partial_ju^j\omega^ie^{\omega\cdot x}\ d\omega dx
-\int_{\mathbb{R}^3}\int_{S^2}\partial_iu^k\partial_ju^j\omega^i\omega^ke^{\omega\cdot x}\ d\omega dx\notag\\
-&\int_{\mathbb{R}^3}\int_{S^2} \partial_ku^k\partial_i\partial_ju^j\omega^i e^{\omega\cdot x}\ d\omega dx\notag\\
=&\int_{\mathbb{R}^3}\int_{S^2} |\nabla\cdot\textbf u|^2e^{\omega\cdot x}\ d\omega dx-\int_{\mathbb{R}^3}\int_{S^{n-1}}\partial_i(\textbf u\cdot\omega)(\nabla\cdot\textbf u)\omega^i e^{\omega\cdot x}\ d\omega dx\notag\\
\leq&\frac{1}{12}\int_{\mathbb{R}^3}\int_{S^2}|\partial\textbf u|^2e^{\omega\cdot x}\ d\omega dx+4\int_{\mathbb{R}^3}\int_{S^2} |\nabla\cdot\textbf u|^2e^{\omega\cdot x}\ d\omega dx\notag\\
\leq&\frac{33}{8}\int_{\mathbb{R}^3}\int_{S^2} |\nabla\cdot\textbf u|^2e^{\omega\cdot x}\ d\omega dx+\frac{1}{12}\int_{\mathbb{R}^3}\int_{S^2} |\nabla\times\textbf u|^2e^{\omega\cdot x}\ d\omega dx\notag\\
+&\frac{1}{8}\int_{\mathbb{R}^3}\int_{S^2} |\textbf u\cdot\omega|^2e^{\omega\cdot x}\ d\omega dx\notag
\end{align}
Using \eqref{sigmaconditions}, for $Q_1(\textbf u,\partial\textbf u)$:
\begin{align}\label{Q1}
&\int_{\mathbb{R}^3}\int_{S^2}Q_1(\textbf u,\partial\textbf u)\cdot\omega e^{\omega\cdot x}\ d\omega dx \\
\geq&  8\lambda(\frac{1}{8}\int_{\mathbb{R}^3}\int_{S^2} |\textbf u\cdot\omega|^2e^{\omega\cdot x}\ d\omega dx
-6\int_{\mathbb{R}^3}\int_{S^2} |\nabla\cdot\textbf u|^2e^{\omega\cdot x}\ d\omega dx \notag\\
-&\frac{1}{12}\int_{\mathbb{R}^3}\int_{S^2} |\nabla\times\textbf u|^2e^{\omega\cdot x}\ d\omega dx)\notag
\end{align}
As for $Q_2(\textbf u,\partial\textbf u)$:
\begin{align}\label{Q2}
&\mid\int_{\mathbb{R}^3}\int_{S^2}Q_2(\textbf u,\partial\textbf u)\cdot\omega e^{\omega\cdot x}\ d\omega dx\mid \\
\leq&\frac{\lambda}{30}\int_{\mathbb{R}^3}\int_{S^2}|\nabla\textbf u|^2 e^{\omega\cdot x}\ d\omega dx\notag\\
\leq&\frac{\lambda}{20}(\int_{\mathbb{R}^3}\int_{S^2}| \nabla\cdot\textbf u|^2e^{\omega\cdot x}\ d\omega dx
+\int_{\mathbb{R}^3}\int_{S^2}| \nabla\times\textbf u|^2e^{\omega\cdot x}\ d\omega dx\notag\\
&+\int_{\mathbb{R}^3}\int_{S^2}| \textbf u\cdot \omega|^2e^{\omega\cdot x}\ d\omega dx)\notag
\end{align}
Combine \eqref{W}-\eqref{1},\eqref{Q1},\eqref{Q2}, we can get eventually:
\begin{align}\label{wiee}
&\int_{\mathbb{R}^3}\int_{S^2}F(\nabla\te u,\nabla^2\te u)\cdot\omega e^{\omega\cdot x}\ d\omega dx\notag\\
\geq& \lambda^2\int_{\mathbb{R}^3}\int_{S^2}|\nabla\cdot\textbf u|^2e^{\omega\cdot x}\ d\omega dx+\lambda\int_{\mathbb{R}^3}\int_{S^2}|\nabla\times\textbf u|^2e^{\omega\cdot x}\ d\omega dx\\
+&\frac{\lambda}{2}\int_{\mathbb{R}^3}\int_{S^2}|\textbf u\cdot \omega|^2e^{\omega\cdot x}\ d\omega dx\notag
\end{align}
Substitute \eqref{wiee} into \eqref{elasticie} :
\begin{align}
&\frac{d}{dt^2}\int_{\mathbb{R}^3}\int_{S^2}e^{\omega\cdot x}(\textbf u\cdot \omega)\ d\omega dx\\
\geq& \int_{\mathbb{R}^3}\int_{S^2}e^{\omega\cdot x}(\textbf u\cdot \omega)\ d\omega dx
+\frac{\lambda}{2}\int_{\mathbb{R}^3}\int_{S^2}|\textbf u\cdot \omega|^2e^{\omega\cdot x}\ d\omega dx\notag
\end{align}
Let
\begin{equation}
X(t)=\int_{\mathbb{R}^3}\int_{S^2}e^{\omega\cdot x}(\textbf u\cdot \omega)\ d\omega dx
\end{equation}
This inequality can be transformed into:
\begin{equation}
X^{''}(t)\geq X(t)+\frac{\lambda X^2(t)}{2e^{t}(t+1)}
\end{equation}

Same argument tells us: With initial data \eqref{3311}, X(t) will blow in finite time.
\end{proof}
\begin{rem}
2D elastodynamics can be treated in the same way.
\end{rem}
\subsection{Magnetohydrodynamics}
3D ideal magnetohydrodynamic can be written as follows:
\begin{equation}\label{MHD3D}
\left \{
\begin{aligned}
&\partial_t\rho+\nabla\cdot(\rho\te u)=0, \\
&\frac{\p (\rho\te u)}{\p t}+\nabla\cdot(\rho\te u\otimes\te u)+\nabla p-\nabla\cdot(\te H\otimes \te H)+\nabla(\frac{1}{2}\te H^2)=0,\\
&\p_t(\frac{\te H}{\rho})+\te (\te u\cdot\nabla)\frac{\te H}{\rho}=(\frac{\te H}{\rho}\cdot\nabla)\te u,\ \ \nabla\cdot\te H=0.
\end{aligned} \right.
\end{equation}

Consider 2D magnetohydrodynamics with special vertical magnetic field ,that is: $\te u(t,x_1,x_2)=(\te u_1,\te u_2,0),\ \te H(t,x_1,x_2)=(0,0,b)$. And assume $p=\frac{\rho^2}{2}$ again, then above-mentioned equations can be reduced to:
\begin{equation}\label{MHD2D}
\left \{
\begin{aligned}
&\partial_t\rho+\nabla\cdot(\rho\te u)=0, \\
&\frac{\p (\rho\te u)}{\p t}+\nabla\cdot(\rho\te u\otimes\te u)+\frac{1}{2}\nabla(\rho^2+ b^2)=0,\\
&\frac{d}{dt}(\frac{b}{\rho})=0.
\end{aligned} \right.
\end{equation}

Take the initial data as:
\begin{align}\label{MHD2Dinitialdata}
t=0:\ \te u=\varepsilon \te u_0,\ b=b_0+\varepsilon h_0,\ \rho=1+\varepsilon\rho_0,
\end{align}
\begin{thm}
For the Cauchy problem \eqref{MHD2D}\eqref{MHD2Dinitialdata}, if the initial data satisfies:
\begin{align}
supp\ \te u_0,h_0,\rho_0&\subset\{x||x|\leq 1\},\ \rho_0\leq0,\ h_0\geq 0.\\
\sqrt{1+b_0^2}\int_{\mathbb R^2}\int_{S^1}e^{\omega\cdot x}\rho_0\ d\omega dx&+\int_{\mathbb R^2}\int_{S^1}e^{\omega\cdot x}(1+\rho_0)(\textbf u_0\cdot \omega)\ d\omega dx>0
\end{align}
then the classical solution blows up in finite time, the lifespan of this solution is less than $C\varepsilon^{-2}$.
\end{thm}
\begin{proof}
From $\frac{d}{dt}(\frac{b}{\rho})=0$, concludes:
\begin{equation}
\frac{b}{\rho}=\frac{b_0+\varepsilon h_0}{1+\varepsilon\rho_0}\geq b_0.
\end{equation}
Multiply $e^{\omega\cdot x}$ and $e^{\omega\cdot x}\omega$ to the first two equations of \eqref{MHD2D}. Then:
\begin{align}
\frac{d}{dt}\int_{\mathbb R^2}\int_{S^1}\rho e^{\omega\cdot x}\ d\omega dx&=\int_{\mathbb R^2}\int_{S^1} \rho(\te u\cdot \omega)e^{\omega\cdot x}\ d\omega dx\\
\frac{d}{dt}\int_{\mathbb R^2}\int_{S^1} \rho(\te u\cdot \omega)e^{\omega\cdot x}\ d\omega dx&=\int_{\mathbb R^2}\int_{S^1} \rho(\te u\cdot \omega)^2e^{\omega\cdot x}\ d\omega dx\\
&+\frac{1}{2}\int_{\mathbb R^2}\int_{S^1}(b^2+\rho^2-b_0^2-1)e^{\omega\cdot x}\ d\omega dx\notag\\
&\geq \int_{\mathbb R^2}\int_{S^1} \rho(\te u\cdot \omega)^2e^{\omega\cdot x}\ d\omega dx\notag\\
&+\frac{1}{2}(b_0^2+1) \int_{\mathbb R^2}\int_{S^1} (\rho^2-1)e^{\omega\cdot x}\ d\omega dx\notag
\end{align}
Still let F(x),\ X(t),\ Y(t) be:
\begin{align}
&F(x)=\int_{S^{1}}e^{\omega\cdot x}d\omega\\
&X(t)=\int_{\mathbb R^2}F(x)(\rho(t,x)-1)\ dx=\int_{\mathbb R^2}\int_{S^1}e^{\omega\cdot x}(\rho(t,x)-1)\ d\omega dx\\
&Y(t)=\int_{\mathbb R^2}F(x)\rho(t,x)(\textbf u\cdot\omega)\ dx=\int_{\mathbb R^2}\int_{S^1}e^{\omega\cdot x}\rho(t,x)(\textbf u\cdot \omega)\ d\omega dx
\end{align}
Direct calculation derives:
\begin{equation}
\int_{|x|\leq at+1}F(x)\ dx\leq  C e^{at+1}(at+1)^{\frac{1}{2}}
\end{equation}
$a^2=b_0^2+1$,\ then:
\begin{align}
\frac{1}{2}(b_0^2+1) \int_{\mathbb R^2}\int_{S^1} (\rho^2-1)e^{\omega\cdot x}\ d\omega dx&\geq a^2X(t)+\frac{1}{2}(b_0^2+1)\frac{CX^2(t)}{a^{\frac{1}{2}}e^{at+1}(t+1)^{\frac{1}{2}}}\\
&\geq a^2X(t)+\frac{CX^2(t)}{e^{at+1}(t+1)^{\frac{1}{2}}}\notag
\end{align}
Eventually, it comes to an ordinary differential inequality:
\begin{equation}
 X^{''}(t)\geq a^2X(t)+\frac{CX^2(t)}{e^{at}(t+1)^{\frac{1}{2}}}
\end{equation}
By Lemma\eqref{OOO} we conclude that X(t) will blow up in finite time, which implies the conclusion of the theorem.
\end{proof}
\section{Acknowledgement}
Yi Zhou was supported by Key Laboratory of Mathematics for Nonlinear Sciences (Fudan University), Ministry of Education of China. Shanghai Key Laboratory for Contemporary Applied Mathematics, School of Mathematical Sciences, Fudan University, P.R. China, NSFC (grants No. 11421061), 973 program (grant No. 2013CB834100) and 111 project.

\bibliographystyle{plain}
\bibliography{ref}



\end{document}